\documentclass[11pt]{amsart}
\usepackage[margin=3cm]{geometry}
\usepackage[T1]{fontenc}
\usepackage{amssymb,mathtools,booktabs}
\usepackage[hidelinks]{hyperref}

\newtheorem{theorem}{Theorem}[section]
\newtheorem{lemma}[theorem]{Lemma}
\newtheorem{proposition}[theorem]{Proposition}
\newtheorem{corollary}[theorem]{Corollary}
\theoremstyle{remark}
\newtheorem{remark}[theorem]{Remark}
\newcommand{\R}{\mathbb R}
\newcommand{\diam}{\operatorname{diam}}
\newcommand{\rad}{\operatorname{rad}}
\newcommand{\conv}{\operatorname{conv}}
\newcommand{\norm}[1]{\left\lVert #1\right\rVert_1}
\newcommand{\ip}[2]{\left\langle #1,#2\right\rangle}

\title[Jung constants of $\ell_1^d$]
  {Upper bounds for Jung constants of $\ell_1^d$}
\author{Xiao Li}
\thanks{Email: \href{mailto:lixiaoms@163.com}{\texttt{lixiaoms@163.com}}.}
\date{}
\keywords{Jung constant, Chebyshev radius, Manhattan distance,
  Hadamard matrix, cut metric}

\begin{document}
\begin{abstract}
For every integer $k\ge1$, we prove that each nonempty bounded set
in real $\ell_1^{4k+1}$ has an ambient enclosing ball of radius at most
$(4k-1)/(4k)$ times its diameter.  The bound is sharp in dimension five:
$J(\ell_1^5)=3/4$.  If a Hadamard matrix of order $4k$ exists, then
$J(\ell_1^{4k-1})=J(\ell_1^{4k})=J(\ell_1^{4k+1})=(4k-1)/(4k)$.
The upper bound for every $k\ge1$ has been formalized in Lean.
\end{abstract}
\maketitle

\section{Introduction}

For a nonempty bounded set $A\subset\R^d$, put
\[
 \norm{x}=\sum_{j=1}^d|x_j|,\qquad
 \diam A=\sup_{x,y\in A}\norm{x-y},\qquad
 \rad A=\inf_{c\in\R^d}\sup_{x\in A}\norm{x-c}.
\]
The Jung constant is
\[
 J(\ell_1^d)=\sup_{\substack{A\ne\varnothing\text{ bounded}\\\diam A>0}}
       \frac{\rad A}{\diam A}.
\]
The center ranges over the whole ambient space.  Using twice the
radius in the definition doubles the values of the Jung constant.

\begin{theorem}\label{thm:main}
Let $k\ge1$ be an integer, $D\ge0$, and let $A\subset\R^{4k+1}$
be nonempty.  Suppose that
\[
 \sum_{j=1}^{4k+1}|x_j-y_j|\le D\qquad(x,y\in A).
\]
There exists $c\in\R^{4k+1}$ such that
\[
 \sum_{j=1}^{4k+1}|x_j-c_j|\le\frac{4k-1}{4k}D
       \qquad(x\in A).
\]
Consequently $J(\ell_1^{4k+1})\le(4k-1)/(4k)$.
\end{theorem}

The contact reduction and coordinate-gap identities reduce the bound
to a cut inequality.  For $k\ge2$, we use the quadratic-form estimate
in Lemma~\ref{lem:cut}.  Section~\ref{sec:five} gives the six-label
estimate needed in dimension five and a set attaining the bound.

\begin{corollary}\label{cor:plateau}
Suppose $k\ge1$ and a Hadamard matrix of order $4k$ exists.  Then
\[
 J(\ell_1^{4k-1})=J(\ell_1^{4k})=J(\ell_1^{4k+1})
      =\frac{4k-1}{4k}.
\]
In particular, for every integer $r\ge2$,
\[
 J(\ell_1^{2^r-1})=J(\ell_1^{2^r})=J(\ell_1^{2^r+1})=1-2^{-r}.
\]
\end{corollary}

The contact reduction, the dimension bound, and the lower bound
obtained from Hadamard matrices are classical.
Dol'nikov~\cite{Dolnikov1987} proved the Hadamard equality criterion
and the estimate $J(\ell_1^{2m})\le(2m-1)/(2m)$.  These results give
the first two equalities in Corollary~\ref{cor:plateau}.
In a 2019 survey, Alimov and Tsar'kov explicitly listed the exact
value of $J(\ell_1^5)$ as unknown~\cite[\S15, p.~831]{AlimovTsarkov2019}.
Ivanov's conjecture, as stated in the survey~\cite[\S15]{AlimovTsarkov2019},
asserts that a Hadamard matrix of order $h=4k$, $k\ge1$, gives the
value $(h-1)/h$ in dimensions $h,h+1,h+2$.
Corollary~\ref{cor:plateau} gives the conjectured value in dimension $h+1$
for $h=4k$.

\section{Contact reduction}

\begin{lemma}\label{lem:contacts}
Let $F\subset\R^d$ be finite and nonempty, with $d\ge1$ and
$\diam F\le1$.  There is a minimizing center $c$ with radius $R$.
There are $q\le d+1$, points $y_i\in F$, signs
$\sigma_i\in\{-1,1\}^d$, and positive weights $\alpha_i$ such that
\begin{align}
 &\sum_{i=1}^q\alpha_i=1,
 \qquad \sum_{i=1}^q\alpha_i\sigma_i=0,\label{eq:balance}\\
 &\ip{\sigma_i}{y_i-c}=\norm{y_i-c}=R.\label{eq:contact}
\end{align}
For every $j$,
\begin{equation}\label{eq:alpha}
 R\le 1-\sum_{\{i:y_i=y_j\}}\alpha_i\le1-\alpha_j.
\end{equation}
The selected points are allowed to repeat.
\end{lemma}

\begin{proof}
For every $x\in\R^d$,
\[
 \norm{x}=\max_{\sigma\in\{-1,1\}^d}\ip{\sigma}{x}.
\]
The continuous function $f(c)=\max_{x\in F}\norm{x-c}$ is coercive,
as $f(c)\ge\norm{x_0-c}$ for any fixed $x_0\in F$, and hence attains
its minimum $R$.
At a minimizing center, let $S$ be the set of signs belonging to the
active affine functions $c\mapsto\ip{\sigma}{x-c}$.
We have $0\in\conv S$.  Indeed, otherwise strict separation of zero
from the compact convex set $\conv S$ gives a vector $u$ with
$\ip{u}{\sigma}>0$ for all $\sigma\in S$.
Moving the center a sufficiently small positive distance in direction
$u$ then strictly decreases all active functions, while all finitely
many inactive functions remain below $R$.  This contradicts minimality.

Choose a convex representation of zero from $S$ with minimal support.
Its coefficients are positive.  The chosen signs are affinely
independent: otherwise a nonzero affine dependence $b_i$ with
$\sum b_i=0$ permits replacement of $\alpha_i$ by
$\alpha_i-tb_i$, where $t=\min_{b_i>0}\alpha_i/b_i$, eliminating
one coefficient while preserving nonnegativity and the representation.
Thus $q\le d+1$.  Choose for each sign an associated active point $y_i$.
Then
$R=\ip{\sigma_i}{y_i-c}\le\norm{y_i-c}\le R$,
which proves~\eqref{eq:contact}.
For any fixed $j$, balance gives
\[
 R=\sum_i\alpha_i\ip{\sigma_i}{y_i-y_j}.
\]
Terms with $y_i=y_j$ vanish; every other inner product is at most
$\norm{y_i-y_j}\le1$.  This proves~\eqref{eq:alpha}.
\end{proof}

\begin{lemma}\label{lem:uniform}
Let $p\ge2$, let $F\subset\ell_1^{2p-1}$ be finite and nonempty
with diameter at most one, and suppose its minimum radius satisfies
\[
 R>\frac{2p-3}{2p-2}.
\]
The contact data in Lemma~\ref{lem:contacts} have
\[
 q=2p,\qquad \sum_{i=1}^{2p}\sigma_i=0.
\]
For every $z\in\R^{2p-1}$,
\begin{equation}\label{eq:universal}
 \sum_{i=1}^{2p}\norm{y_i-z}\ge2pR.
\end{equation}
\end{lemma}

\begin{proof}
Equation~\eqref{eq:alpha} gives
$\alpha_i\le1-R<1/(2p-2)$.
In each coordinate, equation~\eqref{eq:balance} assigns total weight
$1/2$ to each sign class.  A class with at most $p-1$ indices would
have weight strictly less than $(p-1)/(2p-2)=1/2$.
Both classes therefore contain at least $p$ indices.
Since $q\le2p$, each contains exactly $p$ indices and $q=2p$.
This proves the unweighted sign balance.  Consequently
\[
 \sum_i\norm{y_i-z}
 \ge\sum_i\ip{\sigma_i}{y_i-z}
 =\sum_i\ip{\sigma_i}{y_i-c}
   +\ip{\sum_i\sigma_i}{c-z}=2pR.
\]
No identification of the original weights with uniform weights is needed.
\end{proof}

\section{A quadratic-form inequality for cuts}

Let $V$ have $2p$ labels.  Consider a finite family of cuts indexed by
$\mathcal I$, and let $S_A$ be a nonempty proper shore for each
$A\in\mathcal I$.  Different indices may represent the same unordered
cut.  Its size is $s(A)=\min(|S_A|,2p-|S_A|)$, and its semimetric
$\delta_A(i,j)$ is one when the pair crosses the cut and zero otherwise.
An index is balanced when $s(A)=p$.  Sums over $A$ range over
$\mathcal I$.

\begin{lemma}\label{lem:cut}
Let $p\ge3$ be odd.  Give each index $A\in\mathcal I$ a weight
$w_A\ge0$, and write
\[
 d_{ij}=\sum_Aw_A\delta_A(i,j),\qquad
 \Phi=\sum_A s(A)w_A.
\]
Suppose $d_{ij}\le1$ for all distinct $i,j$, and at most $2p-1$
balanced indices have positive weight.  Then
\begin{equation}\label{eq:cutbound}
 \Phi\le B_p:=\frac{p^3(2p-1)}{p^3+p-1}.
\end{equation}
\end{lemma}

\begin{proof}
Set
\[
 T=\sum_{i<j}d_{ij}
   =\sum_A s(A)(2p-s(A))w_A,
 \qquad
 W=\sum_{s(A)=p}w_A,
 \qquad
 c_p=1+\frac{p^2-1}{2p-1}.
\]
All sums over pairs are unordered sums, counting each pair once.
Since an unbalanced cut has $s(A)\le p-1$, we have
\begin{equation}\label{eq:basiccut}
 T\ge(p+1)\Phi-pW,
 \qquad W\le\frac{\Phi}{p}.
\end{equation}
Also $T\le\binom{2p}{2}=p(2p-1)$.

We first show
\begin{equation}\label{eq:keycut}
 T+c_pW\le2p^2.
\end{equation}
If $W=0$, the preceding bound on $T$ proves this immediately.
Suppose $W>0$, and let $\mathcal B$ be the set of positive balanced
indices, of cardinality $t\le2p-1$.  Choose $C\in\mathcal B$ with
$w_C$ maximal.  Then
\[
 W\le t w_C\le(2p-1)w_C.
\]
Choose a shore $S_C$ of size $p$, and put
\[
 v_C=\mathbf1_{S_C}-\mathbf1_{V\setminus S_C},\qquad
 Q_C=-\tfrac12 v_C^{\mathsf T}\mathsf Dv_C,
 \quad \mathsf D=(d_{ij})_{i,j\in V}.
\]
Here $d_{ii}=0$ and $\sum_i(v_C)_i=0$.  For any index $A$, either
choice of shore gives
\[
 v_C^{\mathsf T}\delta_Av_C
 =2\left(\sum_{i\in S_A}(v_C)_i\right)
    \left(\sum_{j\notin S_A}(v_C)_j\right)
 =-2\left(\sum_{i\in S_A}(v_C)_i\right)^2.
\]
It follows that
\begin{equation}\label{eq:squares}
 Q_C=\sum_Aw_A\left(\sum_{i\in S_A}(v_C)_i\right)^2.
\end{equation}
For balanced $A$, the inner sum is
$2|S_A\cap S_C|-p$, an odd integer, so its square is at least one.
For $A=C$, its square is exactly $p^2$.
Unbalanced indices contribute nonnegative squares.  The same statements
hold when another index represents the same cut as $C$.  Using the
maximal choice of $w_C$ gives directly
\begin{equation}\label{eq:lowerQ}
 Q_C\ge W+(p^2-1)w_C
     \ge\left(1+\frac{p^2-1}{2p-1}\right)W=c_pW.
\end{equation}

Let $X_C$ be the sum of $d_{ij}$ over the pairs crossing $C$.
Expanding the quadratic form also gives
\[
 Q_C=X_C-(T-X_C)=2X_C-T\le2p^2-T,
\]
because there are $p^2$ crossing pairs, each of distance at most one.
Together with~\eqref{eq:lowerQ}, this gives
$T+c_pW\le T+Q_C\le2p^2$, proving~\eqref{eq:keycut}.

Finally,
\[
 p-c_p=\frac{(p-1)(p-2)}{2p-1}\ge0.
\]
Equations~\eqref{eq:basiccut} and~\eqref{eq:keycut} imply
\begin{align*}
 2p^2&\ge T+c_pW\\
     &\ge(p+1)\Phi-(p-c_p)W\\
     &\ge\left(p+\frac{c_p}{p}\right)\Phi.
\end{align*}
The last step uses $W\le\Phi/p$ and the displayed sign of $p-c_p$.
Therefore
\[
 \Phi\le\frac{2p^3}{p^2+c_p}
       =\frac{p^3(2p-1)}{p^3+p-1},
\]
including the case $W=0$.
\end{proof}

\begin{remark}\label{rem:comparison}
For every integer $p\ge5$,
\begin{equation}\label{eq:strict}
 \frac{2p(2p-3)}{2p-2}-B_p
 =\frac{p(p^2-5p+3)}{(p-1)(p^3+p-1)}>0.
\end{equation}
Indeed $p^2-5p+3=p(p-5)+3>0$.
For $p=5$, this reads
$B_5=375/43<35/4$, with difference $5/172$.
For $p=3$, the lemma gives $B_3=135/29>9/2$;
the sharper six-label estimate in Section~\ref{sec:five} is needed
for the case $k=1$.
\end{remark}

\section{Coordinate gaps and passage to arbitrary sets}

\begin{lemma}\label{lem:gaps}
Let $y_1,\ldots,y_{2p}\in\R^d$, with repetitions allowed.
For any coordinatewise median $m$, there is a finite indexed family
of nonnegative cut weights, with index set
$\mathcal I=\{1,\ldots,d\}\times\{1,\ldots,2p-1\}$, such that
\[
 \norm{y_i-y_j}=\sum_Aw_A\delta_A(i,j),\qquad
 \sum_{i=1}^{2p}\norm{y_i-m}=\sum_As(A)w_A.
\]
At most $d$ balanced indices have positive weight.
\end{lemma}

\begin{proof}
In a fixed coordinate, order the values as
$a_1\le\cdots\le a_{2p}$, retaining their labels, and put
$g_j=a_{j+1}-a_j\ge0$.
At index $(k,j)$, assign weight $g_j$ to the cut whose shore consists
of the first $j$ labels in coordinate $k$.  The absolute difference for any pair is the sum of the gap
weights separating its labels.  Moreover, since $a_p\le m_k\le a_{p+1}$,
\[
 \sum_{i=1}^{2p}|y_{ik}-m_k|
 =\sum_{i=p+1}^{2p}a_i-\sum_{i=1}^pa_i
 =\sum_{j=1}^{2p-1}\min(j,2p-j)g_j.
\]
Summing over coordinates gives both identities without combining
equal cuts.  Only the index $(k,p)$ is balanced within coordinate $k$,
so there are at most $d$ positive balanced indices.  Ties have zero gap
and may be ordered arbitrarily.  If weights on equal unordered cuts
are combined, the two identities remain unchanged, and the number
of distinct positive balanced cuts is also at most $d$.
\end{proof}

\begin{lemma}\label{lem:extension}
Let $d\ge1$ and $\rho\ge0$.  Suppose every finite nonempty subset of
$\ell_1^d$ of diameter at most one has an ambient enclosing ball of
radius at most $\rho$.  If a nonempty $A\subset\R^d$ satisfies
$\norm{x-y}\le D$ for every $x,y\in A$, where $D\ge0$, then some
$c\in\R^d$ satisfies $\norm{x-c}\le\rho D$ for every $x\in A$.
\end{lemma}

\begin{proof}
If $D=0$, all points coincide and their common value is a suitable
center.  Suppose $D>0$.  Applying the finite assumption to the points
of any finite nonempty $F\subset A$ divided by $D$, and multiplying
the resulting center by $D$, gives an enclosing radius $r=\rho D$.
For each $x\in A$ put
\[
 C_x=\{c\in\R^d:\norm{x-c}\le r\}.
\]
Fix $x_0\in A$.  The sets $C_x\cap C_{x_0}$ are closed subsets of
the compact ball $C_{x_0}$.  Every finite intersection is nonempty:
apply the finite assertion to its indexed points together with $x_0$.
The finite-intersection property gives
$\bigcap_{x\in A}C_x\ne\varnothing$.  A point in this intersection is
the required center.
\end{proof}

\section{The five-dimensional case}\label{sec:five}

The five-dimensional bound rests on a capacity estimate for cuts
on six labels.

\subsection{The six-label capacity estimate}

Let $V=\{1,\ldots,6\}$.  A \emph{cut} is an unordered partition
$C=\{S,V\setminus S\}$ with $\varnothing\ne S\ne V$.  Its size and
cut semimetric are
\[
 s(C)=\min\{|S|,6-|S|\},\qquad
 \delta_C(i,j)=
 \begin{cases}
 1,&\text{$i$ and $j$ are in different shores of $C$},\\
 0,&\text{otherwise}.
 \end{cases}
\]
There are six cuts of size one, fifteen of size two, and ten of size
three.  We call the last ten cuts \emph{balanced}.

\begin{proposition}\label{prop:fivecut}
Assign a nonnegative weight $w_C$ to every cut of $V$, and put
\[
 d_{ij}=\sum_C w_C\delta_C(i,j).
\]
Suppose that $d_{ij}\le1$ for all $i\ne j$ and that at most five balanced
cuts have positive weight.  Then
\begin{equation}\label{eq:fivecutbound}
 \Phi:=\sum_C s(C)w_C\le\frac92.
\end{equation}
\end{proposition}

To describe the balanced supports, fix $6$ as the base label.
For a balanced cut, let $T$ be its unique shore not
containing $6$.  Its complement in $\{1,\ldots,5\}$ is an edge $e$ of
$K_5$.  Thus the ten balanced cuts correspond to the ten edges, with
\begin{equation}\label{eq:five-edge-cut}
 T(e)=\{1,\ldots,5\}\setminus e.
\end{equation}
We abbreviate the edge $\{i,j\}$ by $ij$.

\begin{lemma}\label{lem:five-supports}
Every family of at most five balanced cuts can be extended to a family
of five and then relabeled by a permutation of $V$ so that its
representing edges are one of
\begin{align*}
 A&=\{12,23,34,45,15\},\\
 B&=\{13,14,23,24,34\},\\
 C&=\{12,13,23,14,15\}.
\end{align*}
\end{lemma}

\begin{proof}
There are ten balanced cuts, so a family of fewer than five can be
extended.  Consider its representing graph on five vertices with five
edges.  If disconnected, it must have a component on four vertices:
components of size at most three allow at most
$\binom32+\binom22=4$ edges in total.  It is therefore $K_4$ with one
edge deleted, together with an isolated vertex.

If connected, the graph has a spanning tree with four edges and exactly
one additional edge.  It has a unique cycle, of length three, four, or
five.  A five-cycle uses all vertices; a four-cycle has one attached
leaf.  For a triangle, the two remaining vertices form a length-two
tail, two leaves at one triangle vertex, or leaves at two different
triangle vertices.  These and the disconnected graph are all six
possibilities up to permutations of $\{1,\ldots,5\}$.

It remains to use permutations involving the base.  Exchanging labels
$6$ and $v\in\{1,\ldots,5\}$ fixes a representing edge containing $v$.
An edge $e$ not containing $v$ is replaced by its complementary edge
in $\{1,\ldots,5\}\setminus\{v\}$.  To check this rule, apply the
transposition to the two shores and then choose the shore excluding
$6$ in \eqref{eq:five-edge-cut}.  If $v\notin e$, the new chosen shore is
$e\cup\{v\}$, giving exactly the asserted complement.

The following three applications of this rule pair the six graph types:
\begin{center}
\small
\begin{tabular}{@{}lll@{}}
\toprule
Starting edges & Exchanged base label & Resulting edges\\
\midrule
$12,23,34,45,15$ & $5$ & $34,14,12,45,15$\\
$12,23,34,14,15$ & $1$ & $12,45,25,14,15$\\
$12,23,13,14,45$ & $5$ & $34,14,24,23,45$\\
\bottomrule
\end{tabular}
\end{center}
In the first row, a five-cycle becomes a triangle with leaves at two
different vertices.  In the second, a four-cycle with a leaf becomes
$K_4$ minus an edge and an isolated vertex.  In the third, a triangle
with a length-two tail becomes a triangle with two leaves at one
vertex.  Ordinary permutations identify the representatives of these
three pairs with $A$, $B$, and $C$, respectively.  The base exchanges
are involutions, so the same reductions work in either direction.
\end{proof}

For nonnegative weights $a_{ij}$ on the fifteen unordered pairs of $V$,
define the capacity of a cut by
\[
 \operatorname{cap}_a(S)=\sum_{\substack{i\in S\\j\notin S}}a_{ij}.
\]
This does not depend on the choice of shore.  If
$\deg_a(i)=\sum_{j\ne i}a_{ij}$, then
\begin{equation}\label{eq:five-capacity-degree}
 \operatorname{cap}_a(S)=\sum_{i\in S}\deg_a(i)
             -2\sum_{\{i,j\}\subset S}a_{ij}.
\end{equation}
Expanding the sum gives
\begin{equation}\label{eq:five-duality}
 \sum_{i<j}a_{ij}d_{ij}
   =\sum_C w_C\operatorname{cap}_a(C).
\end{equation}
Thus, if the capacity is at least one on every singleton cut, at least
two on every cut of size two, and at least three on each permitted
balanced cut, then nonnegativity gives
\begin{equation}\label{eq:five-certificate}
 \Phi\le\sum_{i<j}a_{ij}d_{ij}\le\sum_{i<j}a_{ij}.
\end{equation}
Balanced cuts of weight zero impose no condition in this argument.

\begin{proof}[Proof of Proposition~\ref{prop:fivecut}]
Extend the positive balanced support to five cuts, giving added cuts
weight zero, and apply Lemma~\ref{lem:five-supports}.  Relabeling preserves
all pairwise upper bounds, nonnegative weights, and $\Phi$.  We give a
certificate of total weight at most $9/2$ in each representative.

\smallskip
\noindent\emph{Support $A$.}
On the ordinary vertices $1,\ldots,5$, give an edge of the five-cycle
weight $2/15$ and a nonedge weight $8/15$.  Give each pair $i6$ weight
$1/5$.  The total weight is
\[
 5\frac{2}{15}+5\frac{8}{15}+5\frac15=\frac{13}{3}<\frac92.
\]
The weighted degree of each ordinary vertex is $23/15$, and that of
vertex $6$ is one.  Singleton capacities are therefore at least one.
The size-two capacities, according as the pair is a cycle edge, an
ordinary nonedge, or a pair containing the base, are
\[
 2\frac{23}{15}-2\frac{2}{15}=\frac{14}{5},\qquad
 2\frac{23}{15}-2\frac{8}{15}=2,\qquad
 1+\frac{23}{15}-\frac25=\frac{32}{15}.
\]
All are at least two.  For any allowed balanced shore $T(e)$, deleting
the endpoints of the cycle edge $e$ leaves a triple with two cycle
edges and one ordinary nonedge.  By \eqref{eq:five-capacity-degree}, its
capacity is
\[
 3\frac{23}{15}
 -2\left(2\frac{2}{15}+\frac{8}{15}\right)=3.
\]

\smallskip
\noindent\emph{Support $B$.}
Give all fifteen pairs weight $1/4$ and add $3/4$ to pair $56$.
The total is $15/4+3/4=9/2$.  Singleton capacities are at least $5/4$,
and size-two capacities are at least $8/4=2$.  Vertex $5$ is isolated
in the representing graph $B$, so every allowed shore $T(e)$ contains
$5$ and excludes $6$.  Its cut separates the pair $56$, and hence its
capacity is $9/4+3/4=3$.

\smallskip
\noindent\emph{Support $C$.}
Start with weight $1/4$ on every pair and set
\[
 a_{16}=0,\qquad
 a_{12}=a_{13}=a_{46}=a_{56}=\frac12,
\]
leaving all other weights unchanged.  The total is
$15/4-1/4+4/4=9/2$.  Every vertex has weighted degree $3/2$.
Consequently singleton capacities are $3/2$, and every size-two
capacity is
\[
 3-2a_{ij}\ge2.
\]
The five allowed shores are
\[
 345,\quad245,\quad145,\quad235,\quad234,
\]
where a triple denotes its set of labels.  Each has all three internal
pair weights equal to $1/4$.  Its capacity is therefore
$3(3/2)-2(3/4)=3$.

In all three cases, \eqref{eq:five-certificate} proves \eqref{eq:fivecutbound}.
Every pair weight used is nonnegative, so the argument uses only the
upper bounds $d_{ij}\le1$.
\end{proof}

\subsection{The radius bound and sharpness}

\begin{theorem}\label{thm:five}
Every nonempty $A\subset\R^5$ whose pairwise $\ell_1$ distances are
at most $D\ge0$ has an ambient enclosing ball of radius at most
$3D/4$.  Moreover,
\[
 J(\ell_1^5)=\frac34.
\]
\end{theorem}

\begin{proof}
First take a finite nonempty set $F\subset\R^5$ with diameter at most
one and let $R$ be its minimum radius.  If $R>3/4$, apply
Lemma~\ref{lem:uniform} with $p=3$ to obtain six contacts $y_i\in F$
such that for every $z\in\R^5$,
\[
 \sum_{i=1}^6\norm{y_i-z}\ge6R>\frac92.
\]
Choose a coordinatewise median $m$ of these contacts.  Apply
Lemma~\ref{lem:gaps} and combine weights on equal unordered cuts.
The weights preserve all original pairwise distances and the median
functional, and at most five distinct balanced cuts have positive
weight.  Proposition~\ref{prop:fivecut} therefore gives
\[
 \sum_{i=1}^6\norm{y_i-m}=\Phi\le\frac92,
\]
a contradiction.  Thus $R\le3/4$ for every finite diameter-one set.
Lemma~\ref{lem:extension} with $\rho=3/4$ proves the stated upper
bound for all nonempty sets, including $D=0$.

For sharpness, take the four points
\[
 \begin{split}
 x_1&=\tfrac14(1,1,1,0,0),\qquad
 x_2=\tfrac14(1,-1,-1,0,0),\\
 x_3&=\tfrac14(-1,1,-1,0,0),\qquad
 x_4=\tfrac14(-1,-1,1,0,0).
 \end{split}
\]
Every pair of distinct points differs in exactly two coordinates,
so its distance is one.  The zero center has distance $3/4$ from
each point.  In every one of the first three coordinates, two points
have value $1/4$ and two have value $-1/4$.  For any ambient center
$c\in\R^5$,
\begin{align*}
 \sum_{i=1}^4\norm{x_i-c}
 &=2\sum_{j=1}^3\bigl(|c_j-1/4|+|c_j+1/4|\bigr)
     +4(|c_4|+|c_5|)\\
 &\ge3.
\end{align*}
The inequality uses $|t-1/4|+|t+1/4|\ge1/2$.  Hence at least one
point has distance at least $3/4$ from every center.  The four-point
set has diameter one and ambient radius exactly $3/4$, which proves
the matching lower bound.
\end{proof}

\section{Proof of the main theorem}

\begin{proof}[Proof of Theorem~\ref{thm:main}]
For $k=1$, the assertion is exactly the upper bound in
Theorem~\ref{thm:five}.  Suppose henceforth that $k\ge2$ and put
$p=2k+1$, so $p\ge5$ is odd and $d=2p-1$.
Let $F\subset\ell_1^d$ be finite and nonempty with diameter at most
one.  Let $R$ be its minimum radius and suppose
\[
 R>\frac{2p-3}{2p-2}=\frac{4k-1}{4k}.
\]
Lemma~\ref{lem:uniform} gives $2p$ contacts with
\eqref{eq:universal}.  Choose a coordinatewise median $m$ of those
contacts and use Lemma~\ref{lem:gaps}.  The resulting indexed cut
metric has all pairwise distances at most one and at most
$d=2p-1$ positive balanced indices.  Lemma~\ref{lem:cut} and
\eqref{eq:strict} give
\[
 2pR\le\sum_i\norm{y_i-m}=\Phi
 \le B_p<\frac{2p(2p-3)}{2p-2},
\]
a contradiction.  This proves the finite diameter-one assertion.
Apply Lemma~\ref{lem:extension} with
$\rho=(4k-1)/(4k)$ to obtain the conclusion for the given nonempty
$A$ and every $D\ge0$.
\end{proof}

\begin{remark}
The bound $B_p/(2p)$ applies only to the special contact configuration
forced by a radius above the threshold.  It is not an unconditional
upper bound for $J(\ell_1^{2p-1})$; making that substitution would
contradict the Hadamard lower examples.
\end{remark}

\section{Hadamard matrices and sharpness}

We recall the two classical upper bounds.  Lemma~\ref{lem:contacts} has some
$\alpha_j\ge1/q\ge1/(d+1)$, whence
$R\le1-\alpha_j\le d/(d+1)$.
For even $d=2m$, if $R>(2m-1)/(2m)$, every
$\alpha_i<1/(2m)$.  Each sign class of weight $1/2$ would then need
at least $m+1$ indices, contradicting $q\le2m+1$.
Thus
\[
 J(\ell_1^d)\le\frac d{d+1},\qquad
 J(\ell_1^{2m})\le\frac{2m-1}{2m}.
\]
The same scaling and compactness argument extends these finite
estimates to all bounded sets.

\begin{proposition}\label{prop:hadamard}
If a Hadamard matrix of order $h\ge2$ exists, then for every $d\ge h-1$
there is an $h$-point set in $\ell_1^d$ of diameter one and ambient
radius $(h-1)/h$.
\end{proposition}

\begin{proof}
Normalize the first column of the Hadamard matrix to all ones by
changing row signs.  Its rows and columns are orthogonal.
Delete the first column, divide each row by $h$, and append zeros
to obtain points $x_1,\ldots,x_h\in\R^d$.
Two distinct rows differ in exactly $h/2$ entries; none of these is
in the deleted column.  Their distance is thus $(h/2)(2/h)=1$.
Each point has norm $(h-1)/h$.

Every remaining Hadamard column is orthogonal to the constant column,
and therefore has $h/2$ entries of each sign.
For every ambient center $c\in\R^d$,
\begin{align*}
 \frac1h\sum_{i=1}^h\norm{x_i-c}
 &=\frac12\sum_{j=1}^{h-1}
     \bigl(|c_j-1/h|+|c_j+1/h|\bigr)
     +\sum_{j=h}^d|c_j|\\
 &\ge\frac{h-1}{h}.
\end{align*}
For every center, at least one point has distance at least this value.
The zero center attains it for all points, proving the assertion.
\end{proof}

\begin{proof}[Proof of Corollary~\ref{cor:plateau}]
Use Proposition~\ref{prop:hadamard} with $h=4k$ for the three lower
bounds.  The classical dimension bound, the even-dimensional bound,
and Theorem~\ref{thm:main}, respectively, give the matching upper
bounds for every $k\ge1$.

For the unconditional family, start with $H_1=(1)$ and iterate
\[
 H_{2h}=\begin{pmatrix}H_h&H_h\\H_h&-H_h\end{pmatrix}.
\]
If $H_hH_h^{\mathsf T}=hI_h$, direct block multiplication gives
$H_{2h}H_{2h}^{\mathsf T}=2hI_{2h}$.
Thus a Hadamard matrix of order $2^r$ exists for every $r\ge0$.
For $r\ge2$, set $4k=2^r$.
\end{proof}

\section*{Lean formalization}

Lean proofs of Theorem~\ref{thm:main} and the five-dimensional
sharpness result are available at
\url{https://github.com/lixiaoms/Jung}.

\begin{itemize}
\item
\path{lean-toolchain} specifies the Lean version;
\path{lakefile.lean} defines the proof libraries;
\path{lake-manifest.json} fixes the direct and transitive dependency revisions.
\item
\path{Jung.lean} and \path{Jung5.lean} are the library entry points.
\path{Jung/} contains the proof of the general bound, with the statement
for every $k\ge1$ in \path{Jung/Main.lean}.
\path{Jung5/} contains the five-dimensional proof, the finite capacity
certificates, and the sharpness example; its main statements are in
\path{Jung5/Main.lean}.
\item
\path{Audit.lean} prints the types and transitive axiom dependencies
of 151 declarations.  The allowed axioms are
\texttt{propext}, \texttt{Classical.choice}, and \texttt{Quot.sound}.
\item
\path{verify.py} checks the proof and configuration hashes, builds
both libraries, and runs \path{Audit.lean}.  It checks that every
expected axiom report is present and uses only the three axioms listed above.
\item
The audit output is in \path{verification/audit.log}.
\path{verification/source_manifest.json} records the proof and
configuration hashes; \path{verification/summary.json} records the
verified statements and versions.  \path{release_manifest.json}
lists the distributed files and their hashes.
\item
\path{BUILD.md} gives the commands for reproducing the build and audit;
\path{PROOF_GUIDE.md} relates the Lean modules to the proof.
\end{itemize}

\section*{Acknowledgement}
An LLM assisted with the research and preparation of this manuscript.

\end{document}